\documentclass[10pt]{article}
\usepackage{amsmath,amsfonts,latexsym, amssymb}
\usepackage{epsfig}

\usepackage{color}

\newcommand{\e}{\varepsilon}

\newcommand{\rd}{{\rm d}}
\newcommand{\bR}{{\mathbb R}}

\renewcommand{\Re}{{\rm{Re}}}
\renewcommand{\Im}{{\rm {Im}}}

\newcommand{\al}{\alpha}

\newcommand{\be}{\begin{equation}}
\newcommand{\ee}{\end{equation}}

\newcommand{\E}{{\mathbb E }}
\newcommand{\R}{{\mathbb R }}

\newtheorem{theorem}{Theorem}

\newtheorem{proposition}{Proposition}

\newcommand{\qed}{\hfill\fbox{}\par\vspace{0.3mm}}
\newenvironment{proof}{{\bf Proof.}} {\hfill\qed}

\numberwithin{equation}{section}
\numberwithin{theorem}{section}
\numberwithin{definition}{section}
\numberwithin{proposition}{section}
\numberwithin{remark}{section}

\usepackage{geometry}     
\usepackage{graphicx}

\numberwithin{equation}{section}
\def\cal{}

\def\W2{W^{1,2}({\cal O}(M))}

\def\1half{\frac{1}{2}}

\input epsf
\begin{document}

\title{A  comment on the Wigner-Dyson-Mehta bulk universality conjecture for Wigner matrices}

\author{L\'aszl\'o Erd\H os${}^1$\thanks{Partially supported
by SFB-TR 12 Grant of the German Research Council} \
and Horng-Tzer Yau${}^2$\thanks{Partially supported
by NSF grants  DMS-0757425, DMS-0804279} \\
\\ Institute of Mathematics, University of Munich, \\
Theresienstr. 39, D-80333 Munich, Germany${}^1$ \\ \\
Department of Mathematics, Harvard University\\
Cambridge MA 02138, USA${}^2$ \\ \\
\\}

\date{March 28, 2012}

\maketitle

\begin{abstract}
Recently we proved \cite{EPRSY, ERSTVY, ESY4, ESYY, EYY, EYY2, EYYrigi} that  the  eigenvalue correlation functions 
of a general class of random matrices converge,  weakly with respect to the energy,
  to  the corresponding ones of  Gaussian matrices.
Tao and Vu \cite{TV5}  gave a proof   that for the special case of Hermitian Wigner matrices
the convergence 
 can be strengthened
  to  vague  convergence at   any  fixed energy in the bulk.
In this article we
 show that this theorem is 
an immediate corollary of our earlier results. Indeed,  
a more general form 
of this theorem also follows directly   from our work \cite{EKYY2}.

\end{abstract}


\medskip

{\it Running title:} Universality for Wigner matrices

\medskip



\setcounter{section}{1}

\bigskip

Consider an $N\times N$  Wigner ensemble of random matrices 
$H\equiv H_N = (h_{ij})$ with  matrix elements having mean zero and variance $1/N$, i.e., 
\[
   \E\, h_{ij} = 0,   \qquad \E |h_{ij}|^2= \frac{1}{N}, 
 \qquad i,j =1,2,\ldots, N.
\]
A long-standing conjecture of Mehta \cite{M} (also known as the universality conjecture) 
 was to  prove that the local correlation functions of the eigenvalues
depend only on the symmetry class of the ensemble (symmetric, Hermitian or self-dual quaternion).
Tao and Vu  recently published a paper \cite{TV5} where
they renamed   {\it the  special  Hermitian case   of this  conjecture}
as the  ``Wigner-Dyson-Mehta  conjecture". 
They further describe  a key  objective of their  paper  (page 4 of \cite{TV5}) as 
  ``to provide an almost complete solution for vague convergence"
of the  ``Wigner-Dyson-Mehta conjecture".

\medskip

Mehta's  conjecture,   open for almost half a century, was finally 
 solved  for all symmetry classes in \cite{ESY4, ESYY, EYY, EYY2, EYYrigi}.
In these works, convergence was proved in a weak sense in the energy parameter.
Due to a special formula of Br\'ezin-Hikami \cite{BH} and Johansson \cite{J}, 
the Hermitian case is much easier and  was  resolved earlier  in  
\cite{EPRSY, TV, ERSTVY}, in some cases  in a somewhat stronger sense of convergence.
The convergence type, however, was not  considered
a major issue  in Mehta's formulation of the conjecture \cite{M}. 
In \cite{TV5} the authors  single out the much  simpler Hermitian   case as the ``Wigner-Dyson-Mehta conjecture". 
They further
subdivide this  conjecture by different types of convergence and 
address the technical point of strengthening it  to vague convergence at a fixed energy. 
Their one-page proof of this extension 
is  a simple combination of previous facts from \cite{TV} and
our key results from \cite{EPRSY} and \cite{EYYrigi}.
\medskip

We remark that our general
theory can very easily be applied to obtain  universality results  at a fixed energy 
for the Hermitian case. For example,
one  main result of \cite{TV5} (Theorem 5),  stated here  as Theorem \ref{tv},  is  an
 immediate corollary  of two of our previous theorems, see below.
We  also present a more general  result, Theorem \ref{thm:4+e}, 
 for generalized Hermitian Wigner matrices.  
The  proofs of Theorem \ref{tv} and Theorem \ref{thm:4+e}  are only
 a few lines each, given the  key inputs from our earlier works.

\medskip 

We first recall the result  stated as Theorem 5 in  \cite{TV5}.

\begin{theorem}\label{tv}
Let $ p^{(k)}_{N}$ be the $k$-point eigenvalue
correlation function for the Hermitian  Wigner ensemble $ H_N $. Suppose that  a
 sufficiently high  moment of the  rescaled matrix elements  
is  finite,  i.e., 
\be\label{hm}
\E \big|\sqrt{N} h_{ij}\big|^M\le C
\ee
 with some large $M$ and
 $C$, uniformly in $N, i, j$. 
Let $O: \bR^k \to \bR$ 
be a compactly supported bounded  continuous function. Then
for any $|E|<2$  we have
\be\label{sine}
\begin{split}
\lim_{N\to \infty} \int_{\bR^k} O(\al_1, \ldots, \al_k) & 
 \frac{1}{[\varrho(E)]^k}
 p^{(k)}_{N}\Big(E +  \frac{\al_1}{N\varrho(u)},\ldots,
  E + \frac{\al_k}{N\varrho(u)}\Big)
   \rd \al_1\ldots \rd \al_k \\
& = \int_{\bR^k} O(\al_1, \ldots, \al_k)
\det\Big(\frac{\sin \pi(\al_i-\al_j)}{\pi(\al_i-\al_j)}\Big)_{i,j=1}^k
 \rd \al_1\ldots \rd \al_k.
\end{split}
\ee
Here $\varrho(E) = \frac 1 { 2 \pi} \sqrt { 4 - E^2} $ is the limiting eigenvalue density at the energy $E$. 
\end{theorem}

We first recall our general theory for universality \cite{ESY2, EPRSY,  ESY4, EYY,  EYY2, EYYrigi} 
 (see \cite{EY} for a review).
It   consists of three steps: 
 1.~the local semicircle law, 
2.~universality for Gaussian divisible ensembles and 3.~approximation by 
Gaussian divisible ensembles via a perturbation argument. 
This strategy was first introduced  in \cite{EPRSY} and has subsequently 
led  to a recent surge of   activity in the subject. 
In particular, the bulk universality for Hermitian Wigner matrices with smooth distributions, 
at a fixed energy,  was first proved  \cite{EPRSY}.      Notice that all results in 
Steps 1 and 3  hold and were originally stated  for a fixed energy.
Step 2 can be achieved via two different routes:  

\begin{description}
\item[2a] Proposition 3.1 of \cite{EPRSY} (explained   below).
\item[2b] local ergodicity of Dyson Brownian motion (DBM)  (Theorem 4.1 of \cite{ESYY})
\end{description}

The first approach, 2a, which uses an extension of Johansson's
formula \cite{J},  is valid for  fixed energy.  However, it only works  in the  Hermitian case. 
This prompted us to develop the approach 2b which   is very general and  conceptually appealing;  it, however, 
requires an energy average of size $N^{-1 + \e}$. 
To emphasize that our theory is  general in the spirit of universality, 
 we  chose to state our results \cite{EYY, EYY2, EKYY2}  
  for the general cases covering all symmetry classes.  Thus   an average in energy was needed. 
 If we restricted ourselves to the Hermitian case,  we can revert to Step 2a and
  all these results are valid for fixed energy.   
In particular, Theorem \ref{tv} is valid for generalized  Hermitian matrices 
 with finite  $4+ \e$  moments,
see Theorem \ref{thm:4+e} below. 
We now prove Theorem \ref{tv} and demonstrate  how to replace  Step 2b by Step 2a. 

\bigskip 
\noindent   \; {\bf Proof.} 
Recall  that  Gaussian divisible Hermitian matrices are matrices of the form 
\be\label{ha}
 H(a) : = \sqrt {1-a^2}  H_0 + aV
\ee 
where $H_0$ is an arbitrary Hermitian  Wigner matrix, $V$ is an independent  standard 
GUE matrix, and
  $ 0< a < 1$ is a real parameter. 
With this notation,  Proposition 3.1 of \cite{EPRSY} takes  the following form:

\begin{proposition}\label{sinjoh}
Suppose that the matrix  $ H_0$ satisfies the condition \eqref{hm}.  
Then  \eqref{sine} holds for the correlation functions of $H(a)$    if  
$a=  N^{-1/2+\e}$, for any $\e > 0$.  
\end{proposition}

The correlation function comparison theorem,  Theorem 4.2  of \cite{EY} 
 (which is a slight  extension of the Green function comparison theorem,
 Theorem  2.3  of  \cite{EYY}), compares
correlation functions of two Wigner ensembles provided   the first four  
moments  of the matrix elements   are sufficiently close.

\begin{theorem}[Correlation function comparison]\label{comparison}\cite[Theorem 4.2]{EY}  
 Suppose that we have  
two  $N\times N$ Wigner matrices, 
$H^{(v)}$ 
and $H^{(w)}$, with matrix elements $h_{ij}$
given by the random variables $N^{-1/2} v_{ij}$ and 
$N^{-1/2} w_{ij}$, respectively, such that
 $v_{ij}$ and $w_{ij}$ satisfy the  high moment condition \eqref{hm}.
We assume that the first four moments of
  $v_{ij}$ and $w_{ij}$ satisfy, for some $\delta > 0$, that  
\be\label{4}
    \big | \E  (\Re \,  v_{ij})^a (\Im \, v_{ij})^b  -
  \E  (\Re \,  w_{ij})^a (\Im \, w_{ij})^b  \big | \le N^{-\delta -2+ (a+b)/2},
  \qquad 1\le a+b\le 4.
\ee
Let $p_{v, N}^{(k)}$ and $p_{w, N}^{(k)}$
be the  $k-$point correlation functions of the eigenvalues w.r.t. the probability law of the matrix $H^{(v)}$
and $H^{(w)}$, respectively. 
Then for any $|E| < 2$,  any
$k\ge 1$ and  any compactly supported continuous test function
$O:\bR^k\to \bR$ we have   
\be \label{6.3}
\lim_{N\to\infty}\int_{\R^k}  \rd\alpha_1 
\ldots \rd\alpha_k \; O(\alpha_1,\ldots,\alpha_k) 
   \Big ( p_{v, N}^{(k)}  - p_{w, N} ^{(k)} \Big )
  \Big (E+\frac{\alpha_1}{N}, 
\ldots, E+\frac{\alpha_k}{N }\Big) =0.
\ee
\end{theorem}

Theorem \ref{comparison}  concerns  Hermitian matrices under the high moment decay condition \eqref{hm}. 
On the other hand, Theorem 4.2  of \cite{EY} 
was stated for real symmetric matrices under a  subexponential decay condition
for the probability law of its rescaled matrix elements.
The modification needed in the proof to cover   the Hermitian case  
is obvious.    The subexponential condition was used only to obtain 
subexponential decay in $N$ for certain
 large deviation events.
Since decay of order $N^{-C}$ for $C$ large enough
is sufficient for the purpose of proving Theorem \ref{comparison}, the proofs in  \cite{EY, EYY} 
carry through except  changing everywhere   ``subexponential small probability"
to ``with probability $N^{-C}$ for sufficiently large $C$."

Theorem \ref{comparison}  allows us  to compare local
 correlation functions of $H_0$ with $H(a)$ if $\e$ is sufficiently small. 
This concludes the proof of  Theorem \ref{tv}. \qed

 \bigskip

We  emphasized that all our general results are valid for a fixed energy when restricted to Hermitian ensembles. 
As an example, we now state the bulk universality for generalized Hermitian  Wigner ensembles at a fixed energy. 
Recall that   a \emph{generalized Hermitian Wigner matrix} $H = (h_{ij})$ has  independent centered entries with  variances  $ \sigma_{ij}^2 = \E  |h_{ij}|^2 $
satisfying 
\begin{equation*}
\sum_{j} \sigma_{ij}^2 \;=\; 1\,, \qquad \delta  <   N\sigma^2_{ij} <  \delta^{-1} 
\end{equation*}
 for all $i, j$ and  for some  $\delta > 0$  independent of $N$.

\begin{theorem}\label{thm:4+e}
Suppose that $H = (h_{ij})$ is a generalized Hermitian  Wigner matrix. Assume that for some $\gamma > 4$  we have 
\begin{equation} \label{assumptions for 4+e Wigner}
\E \;\bigg| \frac { h_{ij} }{ \sigma_{ij}} \bigg| ^\gamma \;\leq\; C\,,
\end{equation}
for some constant $C$, independent of $i$, $j$, and $N$. Then \eqref{sine} holds for any fixed energy $|E| < 2$.
\end{theorem}

\begin{proof}
This theorem  is the same as  the Hermitian version of 
 Theorem 7.2 in \cite{EKYY2} except that the weak convergence in
energy is now replaced by convergence for any fixed energy in the bulk.   The main idea behind the proof of Theorem 7.2 
is a cutoff argument which compares the original heavy-tailed ensemble  to  one with a 
subexponential decay. 
This cutoff argument holds for any fixed energy.  The energy average  was
 needed in Theorem 7.2 because the universality of the comparison 
ensemble with subexponential decay was stated in the weak sense in  energy so as  to 
apply it to  both symmetry classes.  In order to prove Theorem \ref{thm:4+e}, 
we only have to prove the universality of generalized Hermitian Wigner ensembles with subexponential decay
at a fixed energy.  Inspecting the proof of Theorem \ref{tv} just given above, we only have to check  Proposition~\ref{sinjoh}
(i.e.,  Proposition~3.1 in \cite{EPRSY})
is valid for generalized Wigner matrices as well.

 Along the proof of Proposition 3.1 in \cite{EPRSY}, the properties  of the original Wigner ensemble 
were  used only in the estimate (3.9) which is a direct consequence of a version of the local semicircle law. Since this law was proved for 
generalized Wigner ensembles (see, e.g., Theorem 2.1 in  \cite{EYY}), this extends Proposition 3.1 of  \cite{EPRSY} to the generalized Hermitian Wigner case. 
Together with the cutoff argument in \cite{EKYY2}, this  implies Theorem \ref{thm:4+e}.  
\end{proof}

\end{document}